\documentclass[11pt]{article}
\usepackage{a4wide}

\usepackage[a4paper, margin=2.3cm]{geometry}
\usepackage{amsmath,amssymb,amsthm, comment, enumitem, array}
\usepackage{color}
\usepackage{parskip}
\usepackage{hyperref}
\usepackage{parskip}
\usepackage{setspace}
\usepackage{graphicx}
\usepackage{mathtools}
\usepackage[thinc]{esdiff}
\usepackage[euler]{textgreek}

\usepackage{subfig}

\newtheorem{theorem}{Theorem}[section]

\newtheorem{remark}{Remark}[section]

\newtheorem{lemma}[theorem]{Lemma}

\newtheorem*{definition*}{Definition}

\hypersetup{
    colorlinks=true,
    linkcolor = blue,
    citecolor=magenta,
    urlcolor=cyan,
}

\usepackage[capitalise, nameinlink]{cleveref}
\crefname{equation}{}{}
\crefname{figure}{{\sc Figure}}{{\sc Figure}}
\crefname{subsection}{Subsection}{Subsections}

\def\P{\mathcal{P}}
\def\L{\mathcal{L}}

\begin{document}

\title{An improved point-line incidence bound over arbitrary finite fields via the VC-dimension theory}
\author{Alex Iosevich\thanks{Department of Mathematics, University of Rochester. Email:iosevich@math.rochester.edu}\and Thang Pham\thanks{University of Science, Vietnam National University, Hanoi. Email: phamanhthang.vnu@gmail.com}\and Steven Senger\thanks{Department of Mathematics, Missouri State University. Email: StevenSenger@MissouriState.edu}\and Michael Tait\thanks{Department of Mathematics and Statistics, Villanova University. Email: michael.tait@villanova.edu. }}
\maketitle
\begin{abstract}
The main purpose of this paper is to prove that the point-line incidence bound due to Vinh (2011) over arbitrary finite fields can be improved in certain ranges by using tools from the VC-dimension theory. As consequences, a number of applications will be discussed in detail.
\end{abstract}
\section{Introduction}
Let $\mathbb{F}_q$ be a finite field of order $q$, where $q$ is a prime power. Let $P$ be a set of points in $\mathbb{F}_q^2$ and $L$ be a set of lines in $\mathbb{F}_q^2$. The number of incidences between $P$ and $L$ is denoted by $I(P, L)$. We use the notation $f \ll g$ to mean that $f = O(g)$.

The following point-line incidence bound is due to Vinh \cite{vinh2011szemeredi}, which was proved by using the spectral graph theory method. 
\begin{theorem}\label{vinhthm1}
    The number of incidences between $P$ and $L$ satisfies
    \begin{equation}\label{vinheq1}\left\vert I(P, L)-\frac{|P||L|}{q} \right\vert\ll q^{\frac{1}{2}}\sqrt{|P||L|}.\end{equation}
\end{theorem}

To see how good this theorem is, we compare to the well-known Cauchy-Schwarz estimates. 
\begin{equation}\label{C-S}
    I(P, L)\le \min\{|P|^{1/2}|L|+|P|, ~|P||L|^{1/2}+|L|\}.
\end{equation}
More precisely,
\begin{enumerate}
    \item If $ q^3 = o(|P|L|)$, Theorem \ref{vinhthm1} offers the sharp bound $I(P, L)=(1+o(1))|P||L|/q$.
    \item If $|P|>q$ and $|L|>q$, then the bound \eqref{vinheq1} is better than those in (\ref{C-S}). 
    \item When $|P|<q$ or $|L|<q$ and $|P|\ne |L|$, there is no non-trivial bound in the literature. 
\end{enumerate}
Although there are a number of elementary proofs for the above theorem, for instance, Cilleruelo's proof using Sidon sets in  \cite{cilleruelo2012combinatorial} or a geometric counting argument due to Murphy and Petridis in \cite{murphy2016point}, it is very difficult to say under what conditions the theorem will be improved in the general setting of arbitrary finite fields and what methods will be suitable for this purpose. 

In the setting of prime fields, there exists a number of better results in the literature. The first paper was written by Bourgain, Katz and Tao \cite{BKT} in 2004, in which they proved that for any point set $\P$ and any line set $\L$ in $\mathbb{F}_p^2$ with $|\P|=|\L|=N=p^{\alpha}$, $0<\alpha<2$, we have \begin{equation}\label{eqszt1}I(\P, \L)\ll N^{\frac{3}{2}-\varepsilon},\end{equation} 
where $\varepsilon$ is positive and depends on $\alpha$. The exponent $\varepsilon$ has been made explicit over the years, and the most recent progress is due to Stevens and de~Zeeuw \cite{ffrank}, who showed $I(P, L)\ll |P|^{11/15}|L|^{11/15}$ under a certain conditions on the sizes of $P$ and $L$. This improves earlier results due to Helfgott and Rudnev \cite{HH} and 
Jones \cite{jone}. While the methods in \cite{HH, jone} rely on a pivoting argument and sum-product type problems, the main ingredient in Stevens and De Zeeuw's proof is a point-plane incidence bound due to Rudnev \cite{rudnev2018number}. If we assume to have some additional conditions on the energy of the sets, then some improvements can be obtained, see \cite{petridis2022energy, rudnev2022growth}. We also note that the Stevens-De Zeeuw's result also holds over arbitrary fields, but the bounds obtained are in terms of the characteristic of the field so the sets where the bound is good may be small compared to $q$ if $q$ is not prime. Jones \cite{jones2011explicit} proved that the argument in \cite{HH, jone} can be used to derive a result of the form (\ref{eqszt1}) when $|P|=|L|=N$ over arbitrary finite fields as long as the intersection of the point set $P$ and translations of all subfields satisfies some additional conditions. 

The main purpose of this paper is to show that tools from the VC-dimension theory can be used to prove an improved point-line incidence bound over $\mathbb{F}_q$. A number of applications will also be discussed.
\subsection{Statement of results} 
\begin{theorem}\label{line-incidence}
Let $L$ be a set of lines in $\mathbb{F}_q^2$ of the form $y=ax+b$ with $a\ne 0$ and $P=A\times B\subset \mathbb{F}_q^2$. Assume that $|L||A|>q^{\alpha}\max\{|A|, |L_x|\}$ for some $\alpha\in (0, 1)$, then 
\[I(A\times B, L)\ll \frac{|L||A||B|^{1/2}}{q^{\alpha/2}}+q^{\alpha}|L|^{1/2}|A|^{1/2}|B|^{1/2}.\]
Here $L_x:=\{a\colon y=ax+b\in L\}$. In addition, if $L$ also contains lines of the from $y=c$ or $x=c$, then
\[I(A\times B, L)\ll \frac{|L||A||B|^{1/2}}{q^{\alpha/2}}+q^{\alpha}|L|^{1/2}|A|^{1/2}|B|^{1/2}+2|A||B|.\]
\end{theorem}

The key ingredient in the proof of this theorem is a point-plane incidence theorem in $\mathbb{F}_q^3$, which will be proved by using tools from VC dimension theory. In this paper, we present two versions, where the main difference comes from either the maximal number of collinear points or collinear planes. 

\begin{theorem}\label{plane-inci}
     Let $P$ be a set of points and $\Pi$ be a set of planes in $\mathbb{F}_q^3$. Assume  planes in $\Pi$ are of the form $a\cdot x=1$ with $a\in \mathbb{F}_q^3$. If $|\Pi|\ge 2q^{1+\alpha}$ for some $\alpha\in (0, 1)$, then
    \[I(P, \Pi)\ll \frac{|P||\Pi|}{q^{\alpha}}+|\Pi|q^{2\alpha}.\]
Symmetrically, if $|P|\ge 2q^{1+\alpha}$ for some $\alpha\in (0, 1)$, then
    \[I(P, \Pi)\ll \frac{|P||\Pi|}{q^{\alpha}}+|P|q^{2\alpha}.\]
\end{theorem}
If we relax the condition $|P|\ge  2q^{1+\alpha}$ to $|P|\ge 2kq^{\alpha}$, where $k$ is the maximal number of collinear points in $P$, then we have the next theorem.

\begin{theorem}\label{lightLines}Let $P$ be a set of points and $\Pi$ be a set of planes in $\mathbb{F}_q^3$. Assume  planes in $\Pi$ are of the form $a\cdot x=1$ with $a\in \mathbb{F}_q^3$. If there is no line that contains $k$ points from $P$ and is contained in two planes in $\Pi$, and $|P|\ge 2kq^\alpha$, then
   \[I(P, \Pi)\ll \frac{|P||\Pi|}{q^{\alpha}}+|P|q^{2\alpha}.\]
\end{theorem}

\subsection{Applications}
We now discuss some more applications which might be of independent interest. 

\paragraph{Intersection of planes in $\mathbb{F}_q^3$:} 


The next theorem tells us that given $U'\subset U$ with $|U'|^3\ll |U|^2$, then there are many pairs $(u, v)\in U\times U$ such that the corresponding planes have the same intersection with $U'$, namely, $\pi_u\cap U'=\pi_v\cap U'.$

The formal statement reads as follows. 
\begin{theorem}\label{thm1.9}
Given $U\subset \mathbb{F}_q^3$ with $|U|\gg q^2$, and $U'\subset U$ with $|U'|^3\ll |U|^2$, then the number of pairs  $(u, v)\in U^2$ such that $\pi_u\cap U'=\pi_v\cap U'$ is at least $|U|^2/|U'|^3$.    
\end{theorem}

\paragraph{Distances between two sets in $\mathbb{F}_q^3$:}
Given $E, F\subset \mathbb{F}_q^3$, the distance set between $E$ and $F$
is denoted by $\Delta(E, F)$, which is 
\[\Delta(E, F)=\{||x-y||=(x_1-y_1)^2+(x_2-y_2)^2+(x_3-y_3)^2\colon x\in E, y\in F\}\subset \mathbb{F}_q.\]
When the sizes of $E$ and $F$ are very different, the best current result over arbitrary finite fields is due to Koh and Sun in \cite{KS1}. In particular, they proved that \begin{equation}\label{KSdistances}
|\Delta(E, F)|\gg \begin{cases} \min\{q, |E||F|q^{-(d-1)}\} ~~&\mbox{if}~|E|<q^{\frac{d-1}{2}}\\
\min\{q, |F|q^{-\frac{d-1}{2}}\} ~~&\mbox{if}~q^{\frac{d-1}{2}}\le |E|\le q^{\frac{d+1}{2}}\\
\min\{q, |E||F|q^{-d}\} ~~&\mbox{if}~|E|\ge q^{\frac{d+1}{2}}
\end{cases}.\end{equation}
Note that in even dimensions, we need an additional condition that $|E||F|\gg q^d$.  

In the next theorem, we improve this result in certain ranges in $\mathbb{F}_q^3$.

\begin{theorem}\label{distance}
    Let $E$ and $F$ be sets in $\mathbb{F}_q^3$ with $q\equiv 3\mod 4$. Assume that the number of pairs of distance zero is at most $|E||F|/2$. Define $k$ to be the maximum integer such that there are $k$ collinear points in $F$ $\{u_1,\cdots , u_k\}$ with the property that there is a pair of points $(x,y)\in E\times E$ satisfying $||x-u_i|| = ||y-u_i||\ne 0$ for all $i$. Define $\alpha$ by $|F|>2kq^\alpha$. For $\alpha \in (0,1)$, we have 
    \[|\Delta(E, F)|\gg \begin{cases}
        \max\{k, ~q^\alpha\}~~~&\mbox{if}~|E|\ge q^{3\alpha}\\
        \max\left\lbrace k, \frac{|E|}{q^{2\alpha}} \right\rbrace~~~&\mbox{if}~|E|\le q^{3\alpha}
    \end{cases}\]
\end{theorem}
We refer the reader to \cite{iosevich2007erdos,  murphy2022pinned, shparlinski2006set} for results when the sets $E$ and $F$ are the same. 

\paragraph{Dot product between two sets in $\mathbb{F}_q^3$:}
Given $E, F\subset \mathbb{F}_q^3$, the dot-product set between $E$ and $F$ is denoted by $D(E, F)$, which is 
\[D(E, F)=\{x\cdot y \colon x\in E, y\in F\}\subset \mathbb{F}_q.\]
Similarly, we have the following variant of Theorem \ref{distance} for the dot-product function.
\begin{theorem}\label{dot-product}
    Let $E$ and $F$ be sets in $\mathbb{F}_q^3$. Assume that the number of orthogonal pairs in $E\times F$ is at most $|E||F|/2$, and any two planes of the form $v\cdot x=\lambda$ and $w\cdot x=\lambda$, $v, w\in E$, $\lambda\ne 0$, have at most $k$ collinear points in common from $F$. Assume in addition that $|F|>2kq^{\alpha}$ and $E$ is not contained fully in any plane, then
    \[|D(E, F)|\gg \begin{cases}
        \max\{k, ~q^\alpha\}~~~&\mbox{if}~|E|\ge q^{3\alpha}\\
        \max\left\lbrace k, \frac{|E|}{q^{2\alpha}} \right\rbrace~~~&\mbox{if}~|E|\le q^{3\alpha}
    \end{cases}.\]
\end{theorem}

\subsection{Improved ranges of results}

\paragraph{Improved ranges of Theorem \ref{line-incidence}:} Theorem \ref{line-incidence} performs well when the sizes of the point and line sets are far apart. Relatively less is known about this case compared to the case where the number of points and lines are close together, see for example the discussion in \cite{overflow}. Theorem \ref{line-incidence} offers a better upper bound in its domain compared to those in (\ref{vinheq1}) and (\ref{C-S}) in the range
\begin{enumerate}
    \item If $|L||A|<q^{3\alpha}$, then Theorem \ref{line-incidence} gives $q^\alpha |L|^{1/2}|A|^{1/2}|B|^{1/2}$. This is better than (\ref{vinheq1}) and (\ref{C-S}) in the range
    \[\alpha<1/2, ~q^\alpha\max\{|A|, |L_x|\}<|L||A|<q^{3\alpha}, |L|>q^{2\alpha}, ~|A||B|>|L|q^{2\alpha}.\]
\item If $|L||A|>q^{3\alpha}$, then Theorem \ref{line-incidence} gives $|L||A||B|^{1/2}q^{-\alpha/2}$. This is better than (\ref{vinheq1}) and (\ref{C-S}) in the range
    \[q^\alpha\max\{|A|, |L_x|\}< |L||A|< q^{1+\alpha}, ~|A|<q^\alpha, ~|L|<q^{\alpha}|B|.\]
\end{enumerate}
In the following, we include two examples, one for $|A||L|\ll q^{3\alpha},$ and one for $|A||L|\gg q^{3\alpha}.$
\begin{itemize}
\item ($|A||L|\ll q^{3\alpha}$) Set $\alpha = \frac{1}{4},$ then let $L_x$ be a set of $q^\frac{1}{2}$ different slopes, and pick $q^\frac{1}{8}$ different lines of each slope for a total of $q^\frac{5}{8}$ lines. Then pick any $A,B \subset \mathbb F_q$ such that $|A|=q^\frac{1}{12}$ and $|B|=q^\frac{2}{3}.$
\item ($|A||L|\gg q^{3\alpha}$) Set $\alpha = \frac{2}{5},$ then let $L_x$ be a set of $q^\frac{4}{5}$ different slopes, and pick $q^\frac{1}{5}$ different lines of each slope for a total of $q$ lines. Then pick any $A,B \subset \mathbb F_q$ such that $|A|=q^\frac{4}{15}$ and $|B|=q^\frac{3}{4}.$
\end{itemize}


\paragraph{Improved ranges of Theorem \ref{plane-inci}:}
To see the domains in which the theorem is non-trivial, we recall the following well-known incidence bounds over arbitrary finite fields, namely, Vinh's result in 
\cite{vinh2011szemeredi}:
\begin{equation}\label{vinh}I(P, \Pi)\le \frac{|P||\Pi|}{q}+2q\sqrt{|P||\Pi|},\end{equation}
and the Cauchy-Schwarz incidence bounds:
\begin{equation}\label{C-S2}
    I(P, \Pi)\le q^{1/2}|P|^{1/2}|\Pi|+|P|, ~q^{1/2}|P||\Pi|^{1/2}+|\Pi|.
\end{equation}
\begin{enumerate}
    \item If $|P|\gg q^{3\alpha}$, then Theorem \ref{plane-inci} gives $|P||\Pi|q^{-\alpha}$. This is better than those in (\ref{vinh}) and (\ref{C-S2}) whenever 
    \[q^{3\alpha}<|P|<q^{1+2\alpha}, ~ q^{1+\alpha}<|\Pi|<q^{1+2\alpha}, ~|P||\Pi|<q^{2+2\alpha},~|P||\Pi|\ll q^4, ~\alpha\in (0, 1/2).\]
    \item If $|P|\ll q^{3\alpha}$, then then Theorem \ref{plane-inci} gives $q^{2\alpha}|\Pi|$. This is better than those in (\ref{vinh}) and (\ref{C-S2}) whenever 
    \[ q^{4\alpha-1}<|P|<q^{3\alpha},~~q^{1+\alpha}<|\Pi|<\min\{|P|^2q^{1-4\alpha}, |P|q^{2-4\alpha}\}, |P||\Pi|\ll q^4.\]
\end{enumerate}

To show that Theorem \ref{plane-inci} gives new bounds in non-trivial applications, we include three concrete examples, two for $|P|\ll q^{3\alpha},$ and one for $|P|\gg q^{3\alpha}.$
\begin{itemize}
\item ($|P|\ll q^{3\alpha}$ and $|P|<|\Pi|$) Set $\alpha = \frac{1}{5},$ and let $\Pi$ be a set of $q^\frac{5}{4}$ planes. Then pick any $P \subset \mathbb F_q^3$ such that $|P|=q.$
\item ($|P|\ll q^{3\alpha}$ and $|P|>|\Pi|$) Set $\alpha = \frac{1}{5},$ and let $\Pi$ be a set of $q^\frac{5}{4}$ planes. Then pick any $P \subset \mathbb F_q^3$ such that $|P|=q^\frac{4}{3}.$
\item ($|P|\gg q^{3\alpha}$) Set $\alpha = \frac{1}{3},$ and let $\Pi$ be a set of $q^\frac{3}{2}$ different planes. Then pick any $P \subset \mathbb F_q^3$ such that $|P|=q^\frac{11}{10}.$
\end{itemize}

\paragraph{Improved ranges of Theorem \ref{lightLines}:} To see that Theorem \ref{lightLines} also gives new bounds in non-trivial applications, we include two examples, one for $|\Pi|\ll q^{3\alpha},$ and one for $|\Pi|\gg q^{3\alpha}.$
\begin{itemize}
\item ($|\Pi|\ll q^{3\alpha}$) Set $\alpha = \frac{1}{3},$ and let $\Pi$ be a set of $q^\frac{8}{9}$ planes. Then pick any $P \subset \mathbb F_q^3$ such that $|P|=q^\frac{1}{2},$ taking care that no intersection of planes has more than $k=q^\frac{1}{8}$ points of $P$.
\item ($|\Pi|\gg q^{3\alpha}$) Again, set $\alpha = \frac{1}{3},$ and let $\Pi$ be a set of $q^\frac{5}{4}$ different planes. Then pick any $P \subset \mathbb F_q^3$ such that $|P|=2q^\frac{4}{3}.$ Because distinct planes can share no more than $q$ points, we simply let $k=q$ in this case.
\end{itemize}

\begin{remark}
    We note that over prime fields, a better improvement is due to Rudnev in \cite{rudnev2018number}, which states that 
    \[I(P, \Pi)\ll |P|^{1/2}|\Pi| + k|\Pi|,\]
    where    $k$ is the maximum number of collinear points in $P$, $|P|\le |\Pi|$ and $|P|=O( p^2)$ where $p$ is the characteristic of the field.
\end{remark}


\paragraph{Improved ranges of Theorem \ref{distance}:}
Compared to (\ref{KSdistances}), this result is most effective when the two sets have very different sizes.  
\begin{enumerate}
\item If $|E|\in (q^{2\alpha}, q^{3\alpha})$, $\alpha\in (0, 1/2)$, and $$|F|\in \left(q, ~\min \left\lbrace q^{2-2\alpha}, ~q\cdot \frac{|E|}{q^{2\alpha}} \right\rbrace\right)\cap (2kq^\alpha, q^2),$$
then (\ref{KSdistances}) is weaker than what is offered by Theorem \ref{distance}. More precisely, with $|F|\in (q, q^2)$, applying (\ref{KSdistances}) (and switching the roles of $E$ and $F$) gives us $|\Delta(E, F)|\gg q^{-1}|E|$, which is worse than $q^{-2\alpha}|E|$ when $\alpha\in (0, 1/2)$. If $|E|\in (q^{2\alpha}, q^{3\alpha})\cap (0, q)$, then we have $|\Delta(E, F)|\gg q^{-2}|E||F|$ from (\ref{KSdistances}), which is smaller than $q^{-2\alpha}|E|$ when $|F|<q^{2-2\alpha}$. If $|E|\in (q^{2\alpha}, q^{3\alpha})\cap (q, q^2)$, then we have $|\Delta(E, F)|\gg q^{-1}|F|$, which is also weaker than $q^{-2\alpha}|E|$ when $|F|<q^{1-2\alpha}|E|$. 
\item Similarly, if $|E|\in (q^{2\alpha}, q^{3\alpha})$ and $|F|\in \left(2kq^\alpha, \min \left\lbrace q^{2-2\alpha}, q \right\rbrace\right)$,  then the bound $q^{-2\alpha}|E|$ by Theorem \ref{distance} is also better. If $|E|\leq q$ then \eqref{KSdistances} gives $|E||F|q^{-2}$  and if $|E|\geq q$ \eqref{KSdistances} gives $|F|q^{-1}$; both of these are smaller than $|E|q^{-2\alpha}$ in the range given.
\end{enumerate}
\subsection{Main ideas}
The idea of using the VC-dimension to study incidence questions is not new in the literature, for instance, Fox, Pach, Sheffer, Suk, and Zahl in \cite{fox2017semi} proved the following theorem. 
\begin{theorem}\label{thm1.8}
Let $G=(P, Q, E)$ be a bipartite graph with $|P|=m$ and $|Q|=n$ such that the system $\mathcal{F}_1:=\{N(q)\colon q\in Q\}$ satisfies $\pi_{\mathcal{F}_1}(z)\le cz^d$ for all $z$. Then, if $G$ is $K_{k, k}$-free, then 
\[|E(G)|\le c_1(mn^{1-1/d}+n),\]
where $c_1=c_1(c, d, k)$, and $N(q)$ is the set of neighbors of $q$ in $P$.
\end{theorem}

We first want to compare between the proof of this theorem and our proofs of the above incidence bounds. The strategy in the proof of Theorem \ref{thm1.8} is to do a number of reductions in which they bounded the number of neighbors of a single vertex from $Q$ in each step. More precisely, let $\mathcal{F}_1$ as above, and let $\mathcal{F}_2=\{N(u)\colon u\in P\}$. Given a set $B\in \mathcal{F}_2$ and $k$ points $\{v_1, \ldots, v_k\}\subset Q$. We say that $B$ crosses $\{v_1, \ldots, v_k\}$ if $\{v_1, \ldots, v_k\}\not\subset B$ and $B\cap \{v_1, \ldots, v_k\}\ne \emptyset$. It has been proved in \cite{fox2017semi} that with such families $\mathcal{F}_1$ and $\mathcal{F}_2$, there exist $k$ points $v_1, \ldots, v_k$ in $Q$ such that at most $2c'm/n^{1/d}$ sets from $\mathcal{F}_2$ cross $\{v_1, \ldots, v_k\}$, where $c'=c'(c, d, k)$. For such a $k$-tuples, we observe that the neighborhood of $q_1$ contains at most $c''m/n^{1/d}+(k-1)$ points with $c''=c''(c', k)$. This can be explained from the fact that there are at most $k-1$ points $u_i$ such that $N(u_i)$ covers $\{v_1, \ldots, v_k\}$ (the graph is $K_{k, k}$-free), and the number of sets crossing $\{v_1, \ldots, v_k\}$ is at most $2c'm/n^{1/d}$. Notice that there is a possibility that there are vertices $u_1$ and $u_2$ in $P$ such that $N(u_1)=N(u_2)$ in $\mathcal{F}_2$, which can be ruled out if we assume that each set in $\mathcal{F}_2$ contains at least $k$ points, then it is clear that $c''\le kc'$. After counting the neighbors of $v_1$, we remove it and repeat the argument for the rest until we have less than $k$ points left.

In the finite fields, the above theorem is not practical at least in two and three dimensions. For instance, in two dimensions, if we have a set of points and a set of lines, it is clear that the graph is $K_{2, 2}$-free and $d=2$, thus, it implies nothing but the Cauchy-Schwarz bound. In three dimensions for points and planes, we know that the graph might contain a very large subgraph, say  $K_{q, q}$. Thus, the factor $c_1$ might be very large, which makes the bound worse. Theorem \ref{thm1.8} was improved in \cite{JP} but the limitations when applying to these incidence problems remain the same. 

As mentioned above, to prove Theorem \ref{line-incidence}, we first prove Theorem \ref{lightLines} for points and planes in $\mathbb{F}_q^3$ and then use the fact that if the point set is of Cartesian product structure, then point-line incidences can be reduced to point-plane incidences. To prove Theorem \ref{lightLines}, we first define the graph $G$ with $V(G)=\mathbb{F}_q^3$, and two vertices are connected by an edge if the product between them is $1$. It is not hard to check that the VC-dimension of this graph is at most $3$. Then Lemma \ref{packing} can be rewritten in the way that the number of rich planes, i.e. planes with at least a certain number of points, is bounded. With this information, our theorems follow. 
\section{Preliminaries}
In this section, we recall results related to the VC-dimension which can be found in \cite{fox2017semi}. Let $U$ be a set and $\mathcal{F}$ and family of subsets of $U$. We say that a subset $S\subset U$ is {\em shattered} by $\mathcal{F}$ if for any subset $S'\subset S$, there is some $A\in \mathcal{F}$ such that $A\cap S = S'$. The {\em VC-dimension} of $(U, \mathcal{F})$ is the largest integer $d$ such that there is a $d$-element subset of $S$ which is shattered by $\mathcal{F}$. We also recall the definition of the {\em primal shatter function} of $(U, \mathcal{F})$ which is
\[\pi_{\mathcal{F}}(z)=\max_{U'\subset U, ~|U'|=z}\#\{A\cap U'\colon A\in \mathcal{F}\}.\]
Sauer, Shelah, and others (see \cite{Sauer}, for example) proved that if $\mathcal{F}$ is a set system with VC-dimension $d$, then 
\begin{equation}\label{SS}\pi_{\mathcal{F}}(z)\le \sum_{i=1}^d\binom{z}{i}.\end{equation}
To introduce the next concept, we will need to measure the size of the symmetric difference between two sets, $F_1, F_2\in \mathcal{F}$. We write $|F_1\triangle F_2|$, to mean the size of their symmetric difference, which is given by
\[F_1\triangle F_2=(F_1\cup F_2)\setminus (F_1\cap F_2).\]
A system of sets $\mathcal{F}$ is said to be $(k, \delta)$-separated if for any $k$ sets $F_1, \ldots, F_k\in \mathcal{F}$, we have 
\[|(F_1\cup F_2\cup\cdots\cup F_k)\setminus (F_1\cap F_2\cap \cdots\cap F_k)|\ge \delta.\]
The result we use was called Lemma 2.5 in \cite{fox2017semi} and is known as the packing lemma.
\begin{lemma}\label{packing}
Let $\mathcal{F}$ be a $(k, \delta)$-separated system and $\pi_{\mathcal{F}}(z)\le cz^d$ for all $z$. Then we have $|\mathcal{F}|\le c'(|U|/\delta)^d$ where $c'=c'(k, d, c)$.
\end{lemma}
\section{Proof of Theorem \ref{plane-inci}}
Recall that by assumption, $|\Pi|\geq 2q^{1+\alpha}.$ Now set $U=P$ and $V=\{a\colon \{a\cdot x=1\}\in \Pi\}$. For each $u\in U$, let $N(u)$ be the set of $v\in V$ such that $u\cdot v=1$. 

Then we can write $I(P, \Pi)$ as 
\begin{align*}
     \sum_{u\in U} |N(u)|=&I(P, \Pi)=\sum_{u\colon |N(u)|\le 2|V|q^{-\alpha}}|N(u)|+\sum_{u\colon |N(u)|\ge 2q^{-\alpha}|V|}|N(u)|\\
    &\le \frac{2|U||V|}{q^{\alpha}}+\sum_{i}\sum_{2^i|V|q^{-\alpha}\le |N(u)|<2^{i+1}|V|q^{-\alpha}}|N(u)|.\\
\end{align*}
We now show that the number of $u\in U$ such that $|N(u)|\ge 2^i|V|q^{-\alpha}$ is at most $(q^{\alpha}2^{-i})^3$.
Indeed, let $X_\alpha(U)$ be the set of such $u$. Let $\mathcal{G}$ be the set system {(with ground set $V$)} defined by $\mathcal{G}:=\{N(u)\colon u\in X_\alpha(U)\}$. {Since two planes in $\mathbb{F}_q^3$ either do not intersect or intersect in a line, and there are at most $q$ points on a line}, this implies 
\[N(u_1)\cap N(u_2)\le q, ~\forall~u_1, u_2\in U.\]
Now, combining this with the fact that $|V|>2q^{1+\alpha},$ we see
\[|N(u_1)\triangle N(u_2)|\ge \frac{2^i|V|}{q^\alpha}-q\ge \left(2^i-\frac{1}{2}\right)\frac{|V|}{q^\alpha}.\]
By a direct computation, it is not hard to see that the VC-dimension of $\mathcal{G}$ is at most $3$: if there are $4$ planes which contain a point, then any point in $3$ of them must be in all $4$. Therefore, the inequality (\ref{SS}) holds with $d=3$. Thus, we now apply Lemma \ref{packing} with $\delta = c2^i|V|q^{-\alpha},$ for some constant $c$, to conclude that 
\[|X_\alpha(U)|\ll \left(\frac{|V|}{\delta} \right)^3\ll\left(q^\alpha 2^{-i}\right)^3.\]
So,
\[I(P, \Pi)\ll \frac{|U||V|}{q^{\alpha}}+\sum_{i}|V|q^{2\alpha}\frac{1}{2^{2i}}\leq \frac{|U||V|}{q^{\alpha}}+|V|q^{2\alpha}.\]

By symmetry, we also have 
\[I(P, \Pi)\ll  \frac{|U||V|}{q^{\alpha}}+|U|q^{2\alpha}.\]

\section{Proof of Theorem \ref{lightLines}}

This proof is very similar to that of Theorem \ref{plane-inci}. Set $U=P$ and $V=\{a\colon \{a\cdot x=1\}\in \Pi\}$. For each $v\in V$, let $N(v)$ be the set of $u\in U$ such that $u\cdot v=1$.

Then we can write $I(P, \Pi)$ as 
\begin{align*}
   \sum_{v\in V} |N(v)| &=I(P, \Pi)=\sum_{v\colon |N(v)|\le 2|U|q^{-\alpha}}|N(v)|+\sum_{v\colon |N(v)|\ge 2q^{-\alpha}|U|}|N(v)|\\
    &\le \frac{2|U||V|}{q^{\alpha}}+\sum_{i}\sum_{2^i|U|q^{-\alpha}\le |N(v)|<2^{i+1}|U|q^{-\alpha}}|N(v)|.\\
\end{align*}
As before, we show that the number of $v\in V$ such that $|N(v)|\ge 2^i|U|q^{-\alpha}$ is at most $(q^{\alpha}2^{-i})^3$.
Indeed, let $X_\alpha(V)$ be the set of such $v$. Let $\mathcal{G}$ be the set system defined by $\mathcal{G}:=\{N(v)\colon v\in X_\alpha(V)\}$. Since two planes in $V$ either do not intersect or intersect in a line, and there are at most $k$ collinear points in $U$, this implies 
\[|N(v_1)\cap N(v_2)|\le k, ~\forall~v_1, v_2\in V.\]
Thus, 
\[|N(v_1)\triangle N(v_2)|\ge \frac{2^i|U|}{q^\alpha}-k\ge \left(2^i-\frac{1}{2}\right)\frac{|U|}{q^\alpha},\]
since $|U|>2kq^\alpha$. Again, by a direct computation, it is not hard to see that the VC-dimension of $\mathcal{G}$ is at most $3$: to shatter a set of $4$ points they would need to lie in a common plane and in this case any plane containing $3$ of them must contain all $4$. Therefore, the inequality (\ref{SS}) holds with $d=3$. Thus, we now apply Lemma \ref{packing} to conclude that 
\[|X_\alpha(V)|\ll \left(q^\alpha 2^{-i}\right)^3.\]
So,
\[I(P, \Pi)\ll \frac{|U||V|}{q^{\alpha}}+\sum_{i}|U|q^{2\alpha}\frac{1}{2^{2i}}\leq \frac{|U||V|}{q^{\alpha}}+|U|q^{2\alpha}.\]
The theorem follows.
\section{Proof of Theorem \ref{line-incidence}}
    Without loss of generality, we assume that the lines are defined by the equation $y=ax+b$. Since other lines will contribute at most a factor of $2|P|$ incidences.

    
    By the Cauchy-Schwarz inequality, one has 
    \[I(A\times B, L)\le |B|^{1/2}\left(\#\{(a, b, x, a', b', x')\in L\times A\times L\times A\colon ax+b=a'x'+b'\}\right)^{1/2}.\]
Note that this Cauchy-Schwarz step has also been used in the proof of Stevens-De Zeeuw point-line incidence result in \cite{ffrank}.

The equation $ax+b=a'x'+b'$ can be viewed as an incidence between the point $(x, a', b')$ and the plane defined by $aX-x'Y+Z=-b$ in $\mathbb{F}_q^3$. Let $P$ and $\Pi$ be the sets of corresponding points and planes in $\mathbb{F}_q^3$. We note that any plane in $\Pi$ does not contain any vertical line, so the number of collinear points $k$ in $P$ can be bounded by $k=\max\{|A|, |L_x|\}$ when we do the projection to the $Oxy$ plane. Note that $|P|=|\Pi|=|L||A|$. Thus, if $|L||A|\ge 2kq^\alpha$, then by Theorem \ref{lightLines}, one has 
\[\#\{(a, b, x, a', b', x')\in L\times A\times L\times A\colon ax+b=a'x'+b'\}\lesssim \frac{|L|^2|A|^2}{q^\alpha}+|L||A|q^{2\alpha}.\]
This completes the proof.

\section{Proof of Theorem \ref{thm1.9}}
To prove Theorem \ref{thm1.9}, we need the following lemma. Its proof is standard, but we include it for completeness. 
\begin{lemma}\label{good}
Let $U$ be a set in $\mathbb{F}_q^3$ with $ |U|\geq 8q^2$. Then there exists a subset $U_1$ with $|U_1| \gg |U|$ and for any $u\in U_1$, the number of $u'\in U$ such that $u\cdot u'=1$ is $\Theta\left( \frac{|U|}{q}\right)$. 
\end{lemma}
\begin{proof}
For $u\in U$, let $N(u)$ be the set of $u'\in U$ such that $u\cdot u'=1$.
Define 
\[L:=\left\lbrace u\colon |N(u)|\ge \frac{2|U|}{q}\right\rbrace, ~R:=\left\lbrace u\colon |N(u)|\le \frac{|U|}{2q}\right\rbrace.\]
We identify each point $u\in L$ with the plane defined by $x\cdot u=1$. We abuse the notation to denote the set of corresponding planes by $L$. Then it is clear that $I(U, L)\ge 2|U|^2/q$. Using the point-plane incidence bound (\ref{vinh}), we have 
\[\frac{2|U||L|}{q}\le I(U, L)\le \frac{|U||L|}{q}+2q\sqrt{|U||L|}.\]
This implies $|U||L|\le 4q^4$. Thus, if $|U|\ge 8q^2$, then $|L|\le q^2/2\leq |U|/16$. 

Similarly, we have 
\[\frac{|R||U|}{q}-2q\sqrt{|R||U|}\le I(R, U)\le \frac{|R||U|}{2q}.\]
This infers $|R||U|\le 16q^4$, so $|R|\le 2q^2\le |U|/4$. 

Set $U_1=U\setminus (L\cup R)$, then $|U_1|\gg |U|$, and it satisfies the desired property.
\end{proof}
\begin{proof}[Proof of Theorem \ref{thm1.9}]
Let $U_1$ be the set defined in Lemma \ref{good}, i.e. for all $u\in U_1$, one has $|N(u)| = \Theta\left( |U|q^{-1}\right)$. Let $\mathcal{F}:=\{N(u)\colon u\in U_1\}$. We note that $|\mathcal{F}|=|U_1|\gg |U|$. 

It follows from the definition of $\pi_{\mathcal{F}}(|U'|)$ and \eqref{SS} that
\begin{equation}\label{FU1}
|\mathcal{F}\cap U'|=\#\{A\cap U'\colon A\in \mathcal{F}\}\le \pi_{\mathcal{F}}(|U'|)\ll |U'|^3
\end{equation}
For each $S\in \mathcal{F}\cap U'$, let $m(S)$ be the number of sets $A\in \mathcal{F}$ such that $S=A\cap U'$. We have $\sum_{S}m(S)=|\mathcal{F}|$. By Cauchy-Schwarz, we have 
\begin{equation}\label{FU2}
|\mathcal{F}|^2\le |\mathcal{F}\cap U'|\cdot \sum_{S}m(S)^2.
\end{equation}
Combining \eqref{FU1}, \eqref{FU2}, and the lower bound on $|\mathcal F|,$ we get 
\[\sum_{S}m(S)^2\gg \frac{|U|^2}{|U'|^3},\]
which is what we want to prove.
    
\end{proof}
\section{Proof of Theorem \ref{distance}}
For any $\gamma\in \mathbb{F}_q$, let $R_\gamma = \{(e,f): e\in E, f\in F, ||e-f||=\gamma\}$ and for each $f\in F$ let $R_\gamma(f) = \{e:e\in E, ||e-f|| = \gamma\}$. Additionally, let $T(E, F)$ be the set of triples $(u, v, x)\in E\times E\times F$ such that $||x-u||=||x-v||\ne 0$. By the assumption on the number of pairs at distance $0$, we have that 
\begin{align*}
\frac{|E|^2|F|^2}{4} &\leq \left| \bigcup_{\gamma\not=0} R_\gamma\right|^2 = \left(\sum_{\gamma\not=0} |R_\gamma|\right)^2 \\
&\leq |\Delta(E,F)| \sum_{\gamma\not=0} |R_\gamma|^2 = |\Delta(E,F)| \sum_{\gamma\not=0} \left( \sum_{f} |R_\gamma(f)| \right)^2\\
&\leq |\Delta(E,F)| \sum_{\gamma\not=0} \left( |F| \sum_f |R_\gamma(f)|^2\right) = |\Delta(E,F)|\cdot |F| \cdot |T(E,F)|,
\end{align*}

where both inequalities are by Cauchy-Schwarz. Hence, we have 
\[|\Delta(E, F)|\gg \frac{|E|^2|F|}{|T(E, F)|}.\]
Fix $u\in E$, we observe that for each $v\in E$, the equation $||x-u||=||x-v||$ determines a bisector plane in $\mathbb{F}_q^3$. Let $\Pi_u$ be the set of those planes corresponding to $v\in E$. We now aim to show that each different choice of $v$ leads to a distinct plane in $\Pi_u.$
We note that since $q\equiv 3\mod 4$, any sphere of non-zero radius does not contain any line, i.e. it intersects a line in at most two points. This means that the planes in $\Pi_u$ are distinct. So
\[|T(E, F)|=\sum_{u\in U}I(F, \Pi_u).\]
Under the assumption that $|F| > 2kq^\alpha$, and $k$ is the maximal number of collinear points in $F$ having non-zero distances from any pair of points in $E\times E$, as in the proof of Theorem \ref{lightLines}, we have
\[I(F, \Pi_u)\le \frac{|F||\Pi_u|}{q^\alpha}+q^{2\alpha}|F|.\]
Thus, 
\[|T(E, F)|\ll \frac{|E|^2|F|}{q^\alpha}+q^{2\alpha}|E||F|.\]
So, 
\[|\Delta(E, F)|\gg \min \left\lbrace q^\alpha, \frac{|E|}{q^{2\alpha}}  \right\rbrace.\]
Since $k$ is the maximal number of collinear points in $F$ having non-zero distances from a pair of points in $E\times E$, we also have $|\Delta(E, F)|\ge k/2$. 
Hence, if $|E|\ge q^{3\alpha}$, one has 
\[|\Delta(E, F)|\gg \max\{k, q^\alpha\}.\]
If $|E|\le q^{3\alpha}$, then 
\[|\Delta(E, F)|\gg \max\left\lbrace k, \frac{|E|}{q^{2\alpha}} \right\rbrace.\]
\section{Proof of Theorem \ref{dot-product}}
Because of the assumptions that the number of orthogonal pairs in $E\times F$ is at most $|E||F|/2$, by averaging there is some nonzero $\lambda\in \mathbb{F}_q$ satisfying that the number of pairs in $E\times F$ with $u\cdot v = \lambda$ is at least $|E||F|/|D(E,F)|$. Let $M(E, F)$ be the set of pairs $(u, v)\in E\times F$ such that $u\cdot v=\lambda$. So we have that 
\[|D(E, F)|\gg \frac{|E||F|}{|M(E, F)|},\]

Viewing $E$ as a set of planes and $F$ as a set of points, we apply Theorem \ref{lightLines} to obtain
\[|M(E, F)|\ll \frac{|E||F|}{q^\alpha}+q^{2\alpha}|F|.\]
This gives 
\[|D(E, F)|\gg \min \left\lbrace q^\alpha, \frac{|E|}{q^{2\alpha}} \right\rbrace.\]

We now argue that if we have a line $\ell_0$ with $k$ points from $F$, say $\{u_1, \ldots, u_k\}\subset \ell_0$, then $|D(E, F)|\ge k$. This is more complicated compared to the case of the distance function.

We first observe that for any $\lambda \not=0$, the intersection of the planes defined by $u_i\cdot x = \lambda$ is a line. To see this, consider the intersection of the planes $u_1\cdot x = \lambda$ and $u_2\cdot x = \lambda$ and call this line $\ell_\lambda$. Let $w\in \ell_\lambda$, then it is clear that $u_1, u_2$ belongs to the plane defined by $w\cdot x=\lambda$. Then this plane contains the entire line $\ell_0$. Thus, the line $\ell_\lambda$ is what we want to find. 

The second observation is that if $v\in E$ and $u_i\cdot v=u_j\cdot v=\lambda$, then $u_i\cdot v=\lambda$ for all $1\le i\le k$. This holds for any pair $i\not=j$ in $\ell_0$.

By definition, the lines $\ell_\lambda$ and $\ell_\beta$ are parallel for distinct $\lambda, \beta\ne 0$. Consider the collection of these lines that have at least one point of $E$ on them. Call these lines $\ell_{\lambda_1}, \cdots \ell_{\lambda_t}$. Then $\{\lambda_1,\cdots, \lambda_t\} \in D(E,F)$, and so if $t\geq k$ we are done. If $t < k$, then there must be a $w\in E$ and distinct $u_i$ and $u_j$ such that $w\cdot u_i = w\cdot u_j = \lambda$ for some $\lambda\not=0$. By the observations, we have that the whole line $\ell_0$ is contained in the plane $w\cdot x = \lambda$. Since $E$ is not contained in any single plane, there is some $w'\in E$ which is not contained in this plane, and now $|D(w', \ell_0)|= k$.




In other words, 
\[|D(E, F)|\gg \max\left\lbrace k, ~ \min \left\lbrace q^\alpha, \frac{|E|}{q^{2\alpha}} \right\rbrace\right\rbrace.\]

If $|E|\ge q^{3\alpha}$, one has 
\[|D(E, F)|\gg \max\{k, q^\alpha\}.\]
If $|E|\le q^{3\alpha}$, then 
\[|D(E, F)|\gg \max\left\lbrace k, \frac{|E|}{q^{2\alpha}} \right\rbrace.\]

\section{Acknowledgements}
A. Iosevich was supported, in part, by te HDR Tripods NSF Grant and NSF DMS 2154232 grant. T. Pham would like to
thank to the VIASM for the hospitality and for the excellent working conditions. M. Tait was partially supported by National Science Foundation grant DMS-2011553.

\bibliographystyle{amsplain}

\bibliography{paper}

\end{document}